\documentclass[11pt]{article}
\usepackage{amsmath}
\usepackage{amssymb}
\usepackage{amsthm}

\newtheorem{theorem}{Theorem}
\newcommand{\dif}{\,\mathrm{d}}

\begin{document}

\title{Approximate calculation of operator semigroups by perturbation of generators\footnote{This is the translation from Ukrainian of the communication published in \emph{Dopovidi Natsionalno\u{i}i Akademii Nauk Ukrainy} [Reports Nat. Acad. Sci. Ukr.], 2003, No.11, p. 27--32.}}
\author{A. Yurachkivsky, A. Zhugayevych\\ \\
\textit{Radiophysics Faculty, National Taras Shevchenko University of Kyiv, Ukraine}}
\maketitle

\begin{abstract}
Let $\Omega$ be an operator semigroup with generator $A$ in a sequentially complete locally
convex topological vector space $E$. For a semigroup with generator $A+D$, where $D$ is a
bounded linear operator on $E$, two integral equations are derived. A theorem on continuous
dependence of a semigroup on its generator is proved. An application to random walk on
$\mathbb{Z}$ is given.


\end{abstract}

Let $\Omega$, $\Omega'$ be operator semigroups \cite{Hille,
Yosida,Pazy83} (in this article we consider only one-parameter
semigroups) with generators $A$, $A'$ respectively and let
$D=A'-A$. It is known \cite[p.~77]{Pazy83}, that under appropriate
assumptions
\begin{align}
\Omega' &=\Omega +\Omega D\ast \Omega',  \label{PerturbF1} \\
\Omega' &=\Omega +\Omega'\ast D\Omega ,  \label{PerturbF2}
\end{align}
where $\ast $ signifies convolution, in this case -- of
operator-valued functions. We call these equalities the
\emph{perturbation formulae}.

In the present paper we prove the perturbation formulae under more
general than in \cite{Pazy83} assumptions, then deduce from them
the strongly continuous dependence of a semigroup on its generator
and, finally, apply these results to the approximate calculation
of operator semigroups in coordinate spaces. The idea of the use
of the perturbation formulae is to regard the equalities
(\ref{PerturbF1}) and (\ref{PerturbF2}) as equations for the
unknown function $\Omega'$.

An important example of a semigroup is the transition probability
of Markov process. For this case the perturbation formulae were
discovered and are systematically exploited by Portenko
\cite{Portenko82,Portenko95}. One can find ibidem also a
generalization of these formulae for an inhomogeneous process (its
transition probability creates a two-parameter semigroup).

Denote the space where the semigroups are defined by $E$. Below,
$A_1$ stands for $A'$ and $A_0$ stands for $A$ (likewise
$\Omega_1$ and $\Omega_0$). We impose the usual in the theory of
operator semigroups assumptions:
\renewcommand{\theenumi}{\roman{enumi}}
\renewcommand{\labelenumi}{(\theenumi)}
\begin{enumerate}
\item\label{C1}
$E$ is a sequentially complete locally convex topological vector space;
\item\label{C2}
$A_i$ ($i\in\{0,1\}$) is densely defined linear operator such
that for some $\lambda_0\in\mathbb{R}$ the resolvent
$R_i(\lambda_0)=(\lambda_0\mathbf{1}-A_i)^{-1}$ is defined on the
whole $E$ and the family $\{(\lambda R_i(\lambda))^m:\lambda\geq
1,m\in\mathbb{N}\}$, is equicontinuous.
\end{enumerate}

We remind that the equicontinuity of a family
$\{f_{\theta},\theta\in\Theta\}$ of maps from $E$ to $E$ means
\cite{Yosida} that for any continuous on $E$ seminorm $\|\cdot\|$
there exists a continuous seminorm $\|\cdot\|'$ such that for all
$\theta\in\Theta$ and $q\in E$   $\|f_{\theta}(q)\|\leq\|q\|'$.

\begin{theorem}
Let conditions \textup{(\ref{C1})} and \textup{(\ref{C2})} be fulfilled and the operator $D=A_1-A_0$ be
defined and continuous on the whole $E$. Then for any $t\geq 0$
\begin{align}
\Omega_1(t)&=\Omega_0(t)+\int_0^t\Omega_0(t-\tau)D\Omega_1(\tau)\dif \tau,  \label{Th1_1} \\
\Omega_1(t)&=\Omega_0(t)+\int_0^t\Omega_1(t-\tau)D\Omega_0(\tau)\dif \tau.  \label{Th1_2}
\end{align}
\end{theorem}

\begin{proof}
Since $D$ is defined everywhere,  $A_0$ and $A_1$ have the same
domain which we denote $S$.

Let, at first, condition (\ref{C2}) hold for some $\lambda_0\leq 0$. Then
the Hille--Yosida theorem \cite[IX.7]{Yosida} asserts that,
firstly, the operators $A_i$ generate equicontinuous semigroups
$\Omega_i$ strongly continuously depending on $t$ and, secondly,
$R_i(\lambda)=\int_0^{\infty}e^{-\lambda t}\Omega_i(t)\dif t$.
Consequently, equality (\ref{Th1_1}) holds for all $t\geq 0$ iff
the equality
\begin{equation}
R_1(\lambda)=R_0(\lambda)+R_0(\lambda)D R_1(\lambda)  \label{Th1_3}
\end{equation}
holds for all $\lambda\geq\lambda_0$. To prove the latter we note that by the definition of resolvent
$A_i R_i(\lambda)=\lambda R_i(\lambda)-\mathbf{1}=R_i(\lambda)A_i$ on $S$ so that
$R_0 D R_1 \equiv R_0 A_1 R_1-R_0 A_0 R_1 =R_0(\lambda R_1-\mathbf{1})-(\lambda R_0-\mathbf{1})R_1$.
This proves (\ref{Th1_3}) for the restriction of its both sides on $S$. Assumption (\ref{C2}) allows to extend it to $E$.

In case $\lambda_0>0$ the same argument applied to the operators $\tilde{A}_i=A_i-\lambda_0\mathbf{1}$ yields
(\ref{Th1_1}) for $\tilde{\Omega}_i$ thereafter it
remains to note that, obviously, $\tilde{\Omega}_i(t)=e^{-\lambda_0 t}\Omega_i(t)$.

Swapping $A_1$ and $A_0$, we convert (\ref{Th1_1}) to
(\ref{Th1_2}).
\end{proof}

As the first application of the perturbation formulae we will
prove that the correspondence between a generator and  the
generated semigroup is continuous w.r.t. the strong topology. The
distinction of this result from the Trotter--Kato theorem
(\cite[IX.12]{Yosida}, \cite[IX.2.16]{Kato}) is that the condition
for convergence of a sequence of semigroups is formulated in terms
of generators rather than resolvents.

\begin{theorem}
Assume that: the space $E$ satisfies condition \textup{(\ref{C1})}, and the operators $A_n$ $(n\in\mathbb{Z}_{+})$ satisfy condition \textup{(\ref{C2})}; for each $n\in\mathbb{Z}_{+}$ the operator $B_n\equiv A_n-A$ is defined and bounded
on the whole $E$; for any $q\in E$
\begin{equation}
B_n q\rightarrow 0  \label{Th2_1}
\end{equation}
as $n\rightarrow\infty$. Then for arbitrary $t>0$, $q\in E$ and any seminorm $\|\cdot\|$ among those generating
together the topology of the space $E$
\begin{equation*}
\lim_{n\rightarrow\infty}\sup_{s\leq t}\|\Omega_n(s)q-\Omega(s)q\|=0.
\end{equation*}
\end{theorem}

\begin{proof}
From (\ref{Th2_1}) we deduce by the Banach--Steinhaus theorem \cite[Theorem 2.6]{Rudin} that
\begin{equation}
\sup_n\|B_n\|<\infty.  \label{Th2_2}
\end{equation}

Denote $g_n=\Omega_n q$, $w_n=\|\Omega_n\|$, $W_n(t)=\sup_{s\leq
t}w_n(s)$ (similarly $g$, $w$, $W$), $V_n=\|B_n\|W_n$.
Substituting in (\ref{PerturbF2}) $\Omega'$ by $\Omega_n$ and $D$
by $B_n$, we get by the properties of seminorm
\begin{gather}
w_n(s) \leq w(s)+V_n(s)\int_0^s w_n(\tau)\dif \tau,  \label{Th2_3} \\
\sup_{s\leq t}\|g_n(s)-g(s)\| \leq W_n(t)\int_0^t\|B_n g(\tau)\|\dif \tau.  \label{Th2_4}
\end{gather}

Having iterated inequality (\ref{Th2_3}) once, we get
\begin{align*}
w_n(s) &\leq w(s)+V_n(s)\int_0^s\left(w(\tau)+V_n(\tau)\int_0^{\tau}w_n(p)\dif p\right)\dif\tau \\
&\leq w(s)+V_n(s)\int_0^s w(\tau)\dif \tau+V_n(s)^2\int_0^s(s-p)w_n(p)\dif p.
\end{align*}
Iterating further, we obtain in the limit
\begin{equation*}
w_n(s)\leq w(s)+V_n(s)\int_0^s e^{V_n(s)(s-\tau)}w(\tau)\dif \tau,
\end{equation*}
whence
\begin{equation*}
W_n(t)\leq W(t)e^{\|B_n\|W(t)}.
\end{equation*}
The last equality together with (\ref{Th2_2}) shows that $\sup_n
W_n(t)<\infty$. Now we deduce the conclusion of the theorem from
(\ref{Th2_4}) and (\ref{Th2_1}), relying on the dominated
convergence theorem applicable due to (\ref{Th2_2}).
\end{proof}

The general scheme of applying Theorems 1 and 2 is as follows.
Assume that the semigroup with generator $A_0$ can be calculated
explicitly (in infinite-dimensional spaces this situation is rare
but possible --- a nontrivial example will be given below) and
there exists a strongly converging to the generator $A$ sequence
$(A_n,n\in\mathbb{N})$ of generators such that for each $n$ the
operators $B_n=A_n-A$ and $D_n=A_n-A_0$ (and therefore $A-A_0$)
are defined on the whole space and continuous. If herein the
operators $D_n$ are sufficiently simple (in some sense simpler
than $A-A_0$), then equation (\ref{Th1_1}) or (\ref{Th1_2}) can
prove quite manageable, which will be illustrated below. It
enables us to explicitly calculate the semigroup $\Omega_n$ for
any $n$. Theorem 2 asserts that, for sufficiently large $n$, we
thus obtain a satisfactory approximation of the semigroup
$\Omega$. Unfortunately, this theorem (as well as the
above-mentioned result of Trotter and Kato) does not give the rate
of convergence and, consequently, the exact rule for the choice of
$n$. So, it is expedient, while applying the theorem, to calculate
some initial segment of the sequence $(\Omega_n)$ rather than a
single semigroup. In this case it is simpler to perturb on the
$n$th step $A_{n-1}$ rather than $A$, that is to put
$D_n=A_n-A_{n-1}$.

Let us refine the obtained results for vector topological spaces
of numeral sequences on a finite or countable set $X$ (coordinate
spaces). Let $E$ be some coordinate space and $E'$ be its dual. We
denote likewise a linear continuous operator $A$ on $E$ and its
adjoint operator on $E'$, discerning them by the argument
position: to the right from $A$ in the first case and to the left
in the second one. If $E'$ consists not only of numeral sequences,
but the subspace $E_0'\equiv E'\cap\mathbb{R}^X$ is invariant
w.r.t. the action of the adjoint operator, then we consider the
latter as an operator on this, coordinate restriction of the dual
space. Then both the operator $A$ and its adjoint are determined
by the function on $X^2$ (their \textit{coordinate
representation}):
\begin{align*}
(Aq)(x) &=\sum_y A(x,y)q(y),\ q\in E, \\
(pA)(y) &=\sum_x p(x)A(x,y),\ p\in E_0'.
\end{align*}

In the terminology of differential equations the semigroup
$\Omega$ generated by the operator $A$ is the fundamental
operator, or matrizant, of the system of differential equations
for numeral functions $q(x,\cdot)$:
\begin{equation}
\dot{q}(x,t)=\sum_y A(x,y)q(y,t),\ x\in X,  \label{DiffEq_coord}
\end{equation}
which is tantamount to the single equation
\begin{equation}
\dot{q}=Aq  \label{DiffEq}
\end{equation}
for an $\mathbb{R}^X$-valued function.

Let us find the coordinate representation of the fundamental
operator in the case when $E=l_1(X)$ (the choice of the space will
be explained below) and $A$ is bounded, that is,
\begin{equation}
\|A\|\equiv\sup_y\sum_x|A(x,y)|<\infty.  \label{A_norm}
\end{equation}

Obviously, the solution of equation (\ref{DiffEq}) with initial
condition $q(0)=q_0\in l_1$ is an $l_1$-valued function and
\begin{equation}
\|q(t)\|_1\leq\|q_0\|_1 e^{\|A\|t},  \label{DiffEq_norm}
\end{equation}
where $\|\cdot\|_1$ is the norm in the space $l_1$.

Let $G(\cdot,y,\cdot)$ denote the solution of system
(\ref{DiffEq_coord}) with depending on a parameter $y$ initial
condition
\begin{equation}
G(x,y,0)=\delta(x,y)  \label{G_inicond}
\end{equation}
($\delta$ is Kronecker's function). According to (\ref{DiffEq}) and (\ref{DiffEq_norm})
$\sup_y\sum_x|\dot{G}(x,y,t)|\leq\|A\|e^{\|A\|t}$. This together with (\ref{A_norm}) means that
for any $q_0\in l_1$ the series $\sum_y G(x,y,t)q_0(y)$ admits term-by-term differentiation. So we obtain
the following equation for its sum $q(x,t)$:
\begin{equation*}
\dot{q}(x,t)=\sum_y\dot{G}(x,y,t)q_0(y)=\sum_y\sum_{\xi}A(x,\xi)G(\xi,y,t)q_0(y).
\end{equation*}
Permuting the summations, we become convinced that the functions
$q(x,\cdot)$ satisfy the system of equations (\ref{DiffEq_coord}).
This jointly with the obvious equality $q(x,0)=q_0(x)$ shows that
$G(\cdot,\cdot,t)$ is the coordinate representation of the
operator $\Omega(t)$. We denoted it otherwise than the operator
itself because one can regard the system (\ref{DiffEq_coord}) as a
single differential-difference equation for $q(\cdot,\cdot)$, in
which case $G$ is naturally interpreted as Green's function.

A similar argument shows that $G(x,\cdot,\cdot)$ is the solution
of the conjugate to (\ref{DiffEq_coord}) system
\begin{equation}
\dot{p}(y,t)=\sum_x p(x,t)A(x,y),\ y\in X,  \label{DiffEq2_coord}
\end{equation}
with depending on a parameter $x$ initial condition
(\ref{G_inicond}).

The choice of coordinate space is determined by mathematical or
physical reasons which may be irrelevant to functional analysis.
For example, in probability theory, the differential equations for
semigroups $\dot{\Omega}=A \Omega$, $\dot{\Omega}=\Omega A$
correspond to backward and forward Kolmogorov's equations for the
transition probability matrix of a homogeneous Markov chain. In
this case it is naturally to set $E=l_{\infty}(X)$, where $X$ is
the state space of the chain. Then the dual $E'$ has $l_1$ as its
subspace and $E'_0=l_1$. This choice of the space together with
additional assumptions about $A$ allows to interpret the solution
of (\ref{DiffEq2_coord}) as the one-dimensional distribution of
the chain. In statistical physics, the set (\ref{DiffEq2_coord})
is called (again under some assumptions about $A$) a master
equation and $p$ has the physical meaning of concentration
(density). The latter must be bounded but need not be
summable. In this case $E$ should be matched in such a way that
$E'_0=l_{\infty}$. This is ensured by the choice $E=l_1$ (so that
$E'_0=E'$). In what follows we consider just these spaces.

Denoting
\begin{align}
F(x,y) &=\sum_{\eta}G(x,\eta)D(\eta,y),  \label{def_F} \\
H(x,y) &=\sum_{\xi }D(x,\xi )G(\xi ,y),  \label{def_H}
\end{align}
we  rewrite equalities (\ref{PerturbF1}) and (\ref{PerturbF2}) in
the coordinate form
\begin{align}
G'(x,y) &=G(x,y)+\sum_{\xi } F(x,\xi )\ast G'(\xi ,y),  \label{PerturbF1viaF} \\
G'(x,y) &=G(x,y)+\sum_{\eta}G'(x,\eta)\ast  H(\eta,y).  \label{PerturbF2viaH}
\end{align}
These are countable sets of integral equations with respect to the
functions $G(x,y)$ of argument $t$ suppressed in notation. In the
general case they are not simpler than (\ref{DiffEq_coord}) and
(\ref{DiffEq2_coord}). But if the operator $D$ or its adjoint is
finite-dimensional then in order to solve them it  suffices to
find the resolvent of a finite system of integral equations.
Indeed, denote $\Xi_1=\{\eta:\exists\xi\ D(\xi,\eta)\neq 0\}$,
$\Xi_2=\{\xi:\exists\eta\ D(\xi,\eta)\neq 0\}$, so that
$F(x,\eta)=0$ as $\eta\notin\Xi_1$, $H(\xi,y)=0$ as
$\xi\notin\Xi_2$. Restrict in (\ref{PerturbF1viaF}) the range of
$x$ by $\Xi_1$ and in (\ref{PerturbF2viaH}) the range of $y$ by
$\Xi_2$. Then for every $y$ (\ref{PerturbF1viaF}) is a system of
equations w.r.t. the functions $G'(x,y)$, $x\in\Xi_1$, of argument
$t$; for every $x$ (\ref{PerturbF2viaH}) is a system of equations
w.r.t. $G'(x,y)$, $y\in\Xi_2$. The kernels of the equations and,
consequently, the resolvents do not depend on $y$ in the first
case and $x$ in the second one. Having found one of the two
resolvents, we express via it those functions that enter the
system, thus converting equality (\ref{PerturbF1viaF}) or
(\ref{PerturbF2viaH}) to the explicit formula for the rest of
$G'(x,y)$.

The idea of the use of the perturbation formulae is that,
according to Theorem 2, the solution of an infinite system of
differential equations (\ref{DiffEq_coord}) with bounded $A$ can be
approximated with a desired accuracy by the solution of its finite
subsystem. Even if the operator $A$ is unbounded, but we are able
to solve the set of equations
\begin{equation}
\dot{q}_0(x,t)=\sum_y A_0(x,y)q_0(y,t),\ x\in X, \label{DiffEq0_coord}
\end{equation}
with such an operator $A_0$ that the operator $A-A_0$ is bounded
and both $A$ and $A_0$ satisfy condition (\ref{C2}), then the
solution of (\ref{DiffEq_coord}) can be approximated with a
desired accuracy by the solution of the set of equation finitely
perturbed from (\ref{DiffEq0_coord}). The accuracy can be
evaluated by the perturbation formulae.

As an example we consider a spatially inhomogeneous random walk on
$\mathbb{Z}$. Forward Kolmogorov's equation for the probability
$p(y,t)$ for the walking particle to be at state $y$ at time $t$
has the form (\ref{DiffEq2_coord}) with operator $A$ which has the
following nonzero entries:
\begin{equation*}
A(x,x+1)=\lambda(x),\quad  A(x,x-1)=\mu(x),\quad  A(x,x)=-\lambda(x)-\mu(x),
\end{equation*}
where $\lambda(x)$ and $\mu(x)$ are the intensities of transitions
from $x$ to $x+1$  and $x-1$, respectively (we assume that the
functions $\lambda$ and $\mu$ are  bounded).  For the operator
$A_0$ corresponding to spatially homogeneous walk
($\lambda(x)=\mu(x)=1$) the transition probability is given by the well-known formula  \cite[Formula II.7.7]{Feller}
\begin{equation*}
G_0(x,y,t)=e^{-2t}I_{|x-y|}(2t),
\end{equation*}
where $I_n$ is the modified Bessel function of order $n$. Thus,
for any random walk whose transition intensities are distinct from
unity only for a  finite set of states, we are able to calculate
the transition probabilities exactly. To this end we pass to
Laplace transforms and solve in the Laplace domain the set of
equations (\ref{PerturbF1viaF}) which under the restriction
$x\in\Xi$  will be a finite set of linear equations. The same
formula (\ref{PerturbF1viaF}) will give the answer in Laplace
domain for any $x,y\in\mathbb{Z}$. The inverse Laplace transform
of the result, that is the function $G'(x,y,t)$, turns out to be a
finite linear combination of modified Bessel functions of argument
$2t$ whose orders do not exceed $|x-y|$ (the coefficients are
algorithmically computable).

\end{document}